\documentclass{dcds-bmod}
\usepackage{amsmath,amssymb,amsfonts}
\usepackage{paralist}
\usepackage{mathrsfs}
\usepackage{graphicx}
\usepackage[colorlinks=true]{hyperref}
\hypersetup{urlcolor=blue, citecolor=red}

\def\currentvolume{}
 
  \def\currentyear{}
   \def\currentmonth{}
    \def\ppages{}
     \def\DOI{}

\usepackage[arrow, matrix]{xy}
\usepackage{amsmath}
\usepackage{cancel}
\usepackage{amssymb}
\usepackage{amsthm}
\usepackage[mathscr]{eucal}
\input{epsf}
\theoremstyle{definition}
\newtheorem{definition}{Definition}
\newtheorem{theorem}[definition]{Theorem}
\newtheorem{proposition}[definition]{Proposition}
\newtheorem{lemma}[definition]{Lemma}
\newtheorem{corollary}[definition]{Corollary}
\theoremstyle{remark}
\newtheorem{remark}[definition]{Remark}

\newcounter{enumctr}

\newcommand{\N}{\mathbb{N}}
\newcommand{\R}{\mathbb{R}}
\newcommand{\E}{\mathbb{E}}
\newcommand{\mP}{\mathbb{P}}

\newcommand{\T}{\mathbb{T}}

\newcommand{\cF}{\mathcal{F}}

\newcommand{\cB}{\mathcal{B}}

\newcommand{\eps}{\varepsilon}
\renewcommand{\phi}{\varphi}

\newcommand{\rmd}{\mathrm{d}}

\title[Mean-square dichotomy spectrum for  random dynamical systems]
{The mean-square dichotomy spectrum\\ and a bifurcation to a mean-square attractor}%
\author[Thai Son Doan, Martin Rasmussen, and Peter E. Kloeden]{}


\email{t.doan@imperial.ac.uk}
\email{kloeden@math.uni-frankfurt.de}
\email{m.rasmussen@imperial.ac.uk}

\thanks{The first author was supported by a Marie Curie IEF Fellowship (2013--2015). The second author was supported by an EPSRC Career Acceleration Fellowship (2010--2015). The third author was supported the DFG grant KL~1203/7-1 and a Nelder Visiting Fellowship.}

\begin{document}
\maketitle

\centerline{\scshape Thai Son Doan\footnote{Also at Institute of Mathematics, Vietnam Academy of Science and Technology, 18 Hoang Quoc Viet, Ha Noi, Vietnam}, Martin Rasmussen}
\medskip
{\footnotesize
 \centerline{Department of Mathematics}
 \centerline{Imperial College London}
 \centerline{180 Queen's Gate}
 \centerline{London SW7 2AZ, United Kingdom}}

\medskip

 \centerline{\scshape Peter E. Kloeden}
{\footnotesize
 \centerline{Institut f\"ur  Mathematik, Goethe Universit\"at}
  \centerline{D-60054 Frankfurt am Main, Germany}
}

\bigskip


\begin{abstract}
The dichotomy spectrum is introduced for linear mean-square random dynamical systems, and it is shown that for finite-dimensional mean-field stochastic differential equations, the dichotomy spectrum consists of finitely many compact intervals. It is then demonstrated that a change in the sign of the dichotomy spectrum is associated with a bifurcation from a trivial to a non-trivial mean-square random attractor.
\end{abstract}
\section{Introduction}
Mean-square properties are of traditional interest in the investigation of stochastic systems in engineering and physics. This is quite natural since the Ito stochastic calculus is a mean-square calculus. At first sight, it is thus somewhat surprising that the classical theory of random dynamical systems and their spectra is a pathwise theory, although this can be justified by Doss--Sussman-like  transformations between stochastic differential equations and path-wise random ordinary differential equations \cite{Arnold_98_1}. Such transformations, however, do not apply to mean-field stochastic differential equations, which include expectations of the solution in their coefficient functions \cite{Kloeden_10_1}.

Mean-square random dynamical systems based on deterministic two-parameter semi-groups from the theory of non\-auto\-nomous dynamical systems acting on a state space of random variables or random sets with the mean-square topology were introduced in \cite{Kloeden_12_2}. These act like deterministic systems with the stochasticity built into the state spaces of mean-square random variables.  A mean-square random attractor was defined as a non\-auto\-nomous pullback attractor for such systems from the theory of non\-auto\-nomous dynamical systems \cite{Kloeden_11_2}. The main difficulty in applying the theory is the lack of useful characterisations of compact sets of such spaces of mean-square random variables.

In this paper, a theory of mean-square exponential dichotomies is presented for linear mean-field stochastic differential equations. (It also applies to classical  linear stochastic differential equations). Although the corresponding  mean-square random dynamical systems are essentially  infinite-dimensional their dichotomy spectrum is given by the union of finitely many intervals. This is applied to analyse a nonlinear mean-field stochastic differential equation, for which it is shown that the trivial solution undergoes a mean-square bifurcation leading to a nontrivial mean-square attractor.

The paper is structured as follows. Section~\ref{sec_1} contains the definition of a mean-square random dynamical system, and the notions of mean-square exponential dichotomy and mean-square dichotomy spectrum are introduced. Section~\ref{sec_mssde} explains under which conditions, a mean-field stochastic differential equation generates a mean-square random dynamical system. In Section~\ref{sec_dsmfsde}, the spectral theorem is established, which says that the mean-square spectrum of a linear mean-field stochastic differential equation consists of finitely many compact intervals. Finally, in the last section, it is shown that for a one-dimensional mean-field SDE of pitchfork-type, a stability change in the mean-square spectrum is associated with a bifurcation from a trivial to a non-trivial mean-square random attractor.

\section{Mean-square random dynamical systems}\label{sec_1}

Consider the time set $\R$, and define $\R_{\geq}^2:=\big\{(t,s)\in \R^2: t\geq s\big\}$. Let $(\Omega,\mathcal F, \{\mathcal F_t\}_{t\in\R},\mP)$ be a complete filtered probability space satisfying the usual hypothesis, i.e., ~$\{\mathcal F_t\}_{t\in\R}$ is an increasing and right-continuous family of sub-$\sigma$-algebras of $\mathcal F$, which contain all $\mP$-null sets. Essentially, $\mathcal F_t$ represents the information about the randomness at time $t\in\R$. Finally, define
\[
\mathfrak X:=L^2(\Omega,\mathcal F;\R^d)\quad\mbox{and}\quad
\mathfrak X_t:=L^2(\Omega,\mathcal F_t;\R^d)\qquad\mbox{for } t\in\R
\]
with the norm $\|X\|_{\rm ms}:= \sqrt{\E |X|^2}$, where  $|\cdot|$ is the Euclidean norm on $\R^d$.
\begin{definition}\label{MRDS}
  A \emph{mean-square random dynamical system} (MS-RDS for short) $\phi$ on the underlying phase space $\R^d$ with the filtered probability space $(\Omega,\cF,\{\cF_t\}_{t\in\R},\mP)$ is a family of mappings
  \[
  \phi(t,s,\cdot):\mathfrak X_s\rightarrow \mathfrak X_t,\qquad\mbox{for } (t,s)\in \R^2_{\geq},
  \]
  which satisfies:
  \begin{itemize}
  \item [(1)] \emph{Initial value condition}. $\phi(s,s,X_s)=X_s$ for all $X_s\in\mathfrak X_s$ and $s\in\R$.
  \item [(2)] \emph{Two-parameter semigroup property}. For all $X\in \mathfrak X_{w}$ and all $(t,s), (s,w)\in\R^2_{\geq}$
  \[
  \phi(t,w,X)=\phi(t,s,\phi(s,w,X)).
  \]
  \item [(3)] \emph{Continuity}. $\phi$ is continuous.
  \end{itemize}
\end{definition}

Mean-square random dynamical systems are essentially deterministic with the stochasticity built into or hidden in the time-dependent state spaces.

A MS-RDS $\phi$ is called \emph{linear} if for each $(t,s)\in \R_{\geq}^2$, the map $\phi_{t,s}(\cdot) := \phi(t,s,\cdot)$ is a bounded linear operator. It will be denoted by $\Phi_{t,s}$, and $\Phi_{t,s}(X)$ will conventionally be written $\Phi_{t,s}X$. A spectral theory for linear mean-square random dynamical systems can be established based on exponential dichotomies.

\begin{definition}[Mean-square exponential dichotomy]\label{Exponential Dichotomy}
  Let $\gamma\in\R$. A linear mean-square random dynamical system $\Phi_{t,s}: \mathfrak X_s \to \mathfrak X_t$ is said to admit an \emph{exponential dichotomy} with \emph{growth rate} $\gamma$ if there exist positive constants $K,\alpha$ and a time-dependent decomposition
  \[
  \mathfrak X_t=U_\gamma(t)\oplus S_\gamma(t)\qquad \mbox{for } t\in\R
  \]
  such that
  \begin{align*}
  \|\Phi_{t,s}X_s\|_{\rm ms}&\leq K e^{(\gamma-\alpha)(t-s)}\|X_s\|_{\rm ms}\qquad\hbox{for } X_s\in S_{\gamma}(s) \mbox{ and } t\geq s\,,\\
  \|\Phi_{t,s}X_s\|_{\rm ms}&\geq \frac{1}{K} e^{(\gamma+\alpha)(t-s)}\|X_s\|_{\rm ms}\qquad\hbox{for } X_s\in U_{\gamma}(s)\mbox{ and } t\geq s\,.
  \end{align*}
\end{definition}

A special case of exponential dichotomy, when the growth rate is equal to zero and the space of initial condition consists of the deterministic vectors in $\R^d$, is also investigated in \cite{Ateiwi_02_1,Stoica_10_1}, where a Perron-type condition for existence of this exponential dichotomy is established.

\begin{definition}[Mean-square dichotomy spectrum]\label{DichotomSpectrum}
  The \emph{mean-square dichotomy spectrum} for a linear MS-RDS $\Phi$ is defined as
  \begin{displaymath}
  \Sigma
  :=
  \Big\{\gamma\in\R: \Phi \hbox{ has no exponential dichotomy with growth rate}\; \gamma \Big\}.
  \end{displaymath}
  The set $\rho:=\R\setminus \Sigma$ is called the \emph{resolvent set} of $\Phi$.
\end{definition}

The dichotomy spectrum was first introduced in \cite{Sacker_78_2} for nonautonomous differential equations. Dichotomy spectra for random dynamical systems have been discussed recently in \cite{Callaway_Unpub_1,Cong_02_1,Wang_Unpub_2}.

\section{Mean-field stochastic differential equations}\label{sec_mssde}
Mean-field stochastic differential equations of the form
\begin{equation}\label{Eq1}
  \rmd X_t=f(t,X_t,\E X_t) \,\rmd t+g(t,X_t,\E X_t)\,\rmd W_t
\end{equation}
were introduced in \cite{Kloeden_10_1}. Here $\{W_t\}_{t\in\R}$ is a two-sided scalar Wiener process defined on a probability space $(\Omega,\mathcal F,\mP)$, and $F:=(f,g):\R\times \R^d\times \R^d \rightarrow \R^d\times \R^d$.

Let $\hbox{Lip}(\R^d)$ denote the set of Lipschitz continuous functions $f:\R^d\rightarrow \R^d$, and for each $f\in\hbox{Lip}(\R^d)$, set
\[
\|f(\cdot)\|_{\rm lg}:=\sup_{x\in\R^d}\frac{|f(x)|}{1+|x|}\,.
\]
Suppose that
\begin{itemize}
\item[(A1)]
$
\Gamma:=
\sup_{
(t,x)\in \R\times \R^d}
\Big\{
\hbox{Lip}F(t,x,\cdot)+\|F(t,x,\cdot)\|_{\rm lg}
\Big\}<\infty$.
\item[(A2)] For each $R>0$, there exist a constant $L_R$ and a modulus continuity $\omega_R$ such that
\[
\|F(t_1,x_1,\cdot)-F(t_2,x_2,\cdot)\|^2_{\rm lg}\leq L_R |x_1-x_2|^2+\omega_R(|t_1-t_2|)
\]
for all $(t_k,x_k)\in \R\times \R^d$ with $t_k+|x_k|^2\leq R$, $k\in\{1,2\}$.
\end{itemize}
Let $\{\mathcal F_t\}_{t\in\R}$ be the natural filtration generated by $\{W_t\}_{t\in\R}$, and define
\[
\mathfrak X:=L^2(\Omega,\mathcal F;\mathbb R^d),\quad \mathfrak X_t:=L^2(\Omega,\mathcal F_t;\mathbb R^d)\quad\hbox{ for } t\in\R.
\]
Given any initial condition $X_s\in\mathfrak X_s$, $s\in\R$, a solution of \eqref{Eq1} is a stochastic process $\{X_t\}_{t\geq s}$ with $X_t\in \mathfrak X_t$ for $t\geq s$, satisfying the stochastic integral equation
\[
X_t=X_s+\int_s^t f(u,X_u,\E X_u)\,\rmd u+\int_s^t g(u,X_u,\E X_u)\,\rmd W_u\,.
\]
It was shown in \cite{Kloeden_10_1} that the SDE~\eqref{Eq1} has a unique solution and generates a MS-RDS $\{\phi_{t,s}\}_{t\geq s}$ on the underlying phase space $\R^d$ with a probability set-up $(\Omega,\mathcal F,\{\mathcal F_t\}_{t\in\R},\mP)$, defined by $\phi_{t,s}:\mathfrak X_s\rightarrow \mathfrak X_t$ with
\[
\phi_{t,s}(X_s)=X_t\qquad\hbox{ for  } X_s\in \mathfrak X_s\,.
\]

\section{Mean-square dichotomy spectrum for linear mean-field stochastic differential equations}\label{sec_dsmfsde}
Consider a linear mean-field stochastic differential equation
\begin{equation}\label{Eq2}
\rmd X_t=\big(A(t)X_t+B(t)\E X_t\big)\,\rmd t+\big(C(t)X_t+D(t)\E X_t\big)\,\rmd W_t\,,
\end{equation}
where $A,B,C,D:\R\rightarrow \R^{d\times d}$ are continuous bounded functions, which generates a linear mean-square random dynamical system $\Phi_{t,s}$.

\begin{proposition}[Equations for the first and second moments]\label{Prp3}
  Let $X_s\in \mathfrak X_s$, and define $X_t:=\Phi_{t,s}  X_s$ for $t\geq s$. Then $\frac{\rmd}{\rmd t}\E X_t= \big(A(t)+B(t)\big)\E(X_t)$ and for all $i,j\in \{1,\dots,d\}$,
  \begin{align*}
    \frac{\rmd}{\rmd t}\E X^i_t X^j_t&
    =
    \sum_{k=1}^d\left( a_{ik}(t)\E X_t^kX_t^j+a_{jk}(t)\E X_t^kX_t^i\right)\\
    &
    +\sum_{m,n=1}^dc_{im}(t)c_{jn}(t)\E X_t^mX_t^n
    +\sum_{k=1}^d \left(b_{ik}(t)\E X_t^k\E X_t^j+b_{jk}(t)\E X_t^k\E X_t^i\right)\\
    &+
    \sum_{m,n=1}^d \left(c_{im}(t)d_{jn}(t)+c_{jn}(t)d_{im}(t)+d_{im}(t)d_{jn}(t)\right)\E X_t^m\E X_t^n.
  \end{align*}
\end{proposition}

\begin{proof}
From \eqref{Eq2},
\[
\rmd X^i_t=\sum_{k=1}^d\left(a_{ik}(t)X^k_t+b_{ik}(t)\E X^k_t\right)\,\rmd t+\sum_{k=1}^d\left(c_{ik}(t)X^k_t+d_{ik}(t)\E X^k_t\right)\,\rmd W_t
\]
holds for $t\in \R$. Taking the expectation of variables in both sides gives
\[
\E X_t=\E X_s+\int_s^t \big(A(u)+B(u)\big)\E X_u\,\rmd u\quad\hbox{for } t\geq s,
\]
which proves the first statement. Ito's product formula \cite[Example 3.4.1]{Kloeden_92_1} and the expectation then yield the second statement.
\end{proof}

\begin{corollary}\label{Corollary1}
Let $(U_{i,j}(t,s),V_{i,j}(t,s))$ be the evolution operator of the linear nonautonomous differential equation in $\R^{d(d+1)}$
\begin{align}
\dot u_{i,j}
=&
\sum_{k=1}^d \big( a_{ik}(t)+ b_{ik}(t)\big)u_{k,j}+ \big( a_{jk}(t)+ b_{jk}(t)\big) u_{k,i}\notag\\
\dot v_{i,j}
=&
\sum_{k=1}^d\left( a_{ik}(t)v_{k,j}+a_{jk}(t)v_{k,i}\right)+\sum_{m,n=1}^dc_{im}(t)c_{jn}(t)v_{m,n}\notag\\
&+\sum_{k=1}^d \left(b_{ik}(t)u_{k,j}+b_{jk}(t)u_{k,i}\right)\label{LinearNDEb}\\
&+\sum_{m,n=1}^d \left(c_{im}(t)d_{jn}(t)+c_{jn}(t)d_{im}(t)+d_{im}(t)d_{jn}(t)\right)u_{m,n},\notag
\end{align}
where $1\leq i\leq j\leq d$. Then for any $X_s\in\mathfrak X_s$,
\begin{equation}\label{Eq10}
\|\Phi_{t,s}X_s\|_{\rm ms}=\left(\sum_{i=1}^d V_{i,i}(t,s)(\pi_sX_s)\right)^{\frac{1}{2}},
\end{equation}
where the map $\pi_s=\pi_s^1\times \pi_s^2:\mathfrak X_s\rightarrow \R^{\frac{d(d+1)}{2}}\times \R^{\frac{d(d+1)}{2}} $ is defined by
\[
(\pi_s^1X_s)_{i,j}=\E X_s^i\E X_s^j\quad\text{and}\quad (\pi_s^2X_s)_{i,j}=\E X_s^iX_s^j\,.
\]
\end{corollary}
\begin{remark}
  It is interesting to compare the above equations with the ordinary differential equations for the first moment and second moments of linear stochastic differential equations, see e.g.~\cite[Section 6.2]{Khasminskii_12_1}.
\end{remark}
The proofs of the following preparatory results  are straightforward.
\begin{lemma}\label{Lemma_1}
  Let $\gamma\in\R$  be such that  that the linear MS-RDS $\Phi_{t,s}:\mathfrak X_s\rightarrow \mathfrak X_t$ generated by \eqref{Eq1}  admits an exponential dichotomy with the growth rate $\gamma$ and a decomposition
  \[
  \mathfrak X_t=U_\gamma(t) \oplus S_\gamma (t).
  \]
  Then the subspace $S_\gamma (t)$ is uniquely determined, i.e.,~if the linear MS-RDS $\Phi_{t,s}$ also admits an exponential dichotomy with the growth rate $\gamma$ and another decomposition
  \[
  \mathfrak X_t=\widehat U_\gamma(t) \oplus  \widehat S_\gamma (t),
  \]
  then $S_\gamma (t)=\widehat S_\gamma (t)$ for all $t\in\R$.
\end{lemma}
The subspaces $S_\gamma(t)$ of an exponential dichotomy with the growth rate $\gamma$  are its  stable subspaces. The following lemma provides  an inclusion relation between these stable subspaces. Its  proof follows directly from the definition of an exponential dichotomy.
\begin{lemma}\label{Lemma_2}
  Let $\gamma_1<\gamma_2$ be such that the linear MS-RDS admits an exponential dichotomy with the growth rates  $\gamma_1$ and $\gamma_2$. Then $S_{\gamma_1}(t)\subset S_{\gamma_2}(t)$ for all $t\in\R$.
\end{lemma}
One of the main results of this paper is the following characterisation of the dichotomy spectrum.
\begin{theorem}[Spectral Theorem]\label{SpectralTheory}
  Suppose that the coefficient functions in the linear mean-field stochastic differential equation \eqref{Eq2} satisfy
  \begin{equation}\label{BoundedCondition}
    \max\big\{|a_{ij}(t)|, |b_{ij}(t)|, |c_{ij}(t)|,|d_{ij}(t)|\big\}\leq m \qquad\hbox{for } i,j\in\{1,\dots,d\}\mbox{ and } t\in\R
  \end{equation}
  with some $m>0$. Then the dichotomy spectrum $\Sigma$ is the disjoint union of at most $d(d+1)$ compact intervals $[a_1,b_1]$, $\dots$, $[a_n,b_n]$ with $a_1 \leq b_1 <  a_2 \leq b_2 \leq \dots < a_n \leq b_n$. Furthermore, for each $s\in \R$, there exists a filtration of subspaces
  \[
  \{0\}\subsetneq V_1(s)\subsetneq V_2(s)\subsetneq\dots\subsetneq V_n(s)=\mathfrak X_s
  \]
  which satisfies that for any $i\in\{1,\dots,n\}$, a random variable $X_s\in V_i(s)$ if and only if for any $\eps>0$, there exists $K(\eps)>0$ such that
  \[
  \|\Phi_{t,s}X_s\|_{\rm ms}\leq K(\eps) e^{(b_i+\eps)(t-s)}\quad \hbox{for } t\geq s\,.
  \]
\end{theorem}

\begin{proof}
  The proof is divided into several steps.\\
  \noindent
  \emph{Step 1.} First it will be shown that $(-\infty,-\Gamma)\subset\rho$ and $(\Gamma,\infty)\subset \rho$, where $\Gamma := 2 dm +2d^2m^2$. Let $X_s\in \mathfrak X_s$ be arbitrary, and define
  \[
  \alpha(t):=\max_{i,j\in\{1,\dots,d\}}\big\{\E X^i_t X^j_t, \E X^i_t \E X^j_t\big\}\quad\hbox{for } t\geq s.
  \]
  By the inequalities $\E XY\leq \sqrt{\E X^2 \E Y^2}$ and $(\E X)^2\leq \E X^2$, it follows that
  \[
  \alpha(t)=\max_{1\leq i\leq d} (\E X_t^i)^2\leq \|\Phi(t,s)X_s\|^2_{\rm ms}.
  \]
  Then by Corollary~\ref{Corollary1},
  \[
    \alpha(t)\leq \alpha(s)+2\Gamma \int_s^t  \alpha(u)\,\rmd u\,,
  \]
  and Gronwall's inequality then yields
  \[
  \|\Phi(t,s)X_s\|^2_{\rm ms}\leq d\alpha(t)\leq de^{2\Gamma(t-s)}\alpha(s)\leq de^{2\Gamma(t-s)}\|X_s\|^2_{\rm ms}.
  \]
  This proves that $(\Gamma,\infty) \subset \rho$. Time reversal of the equations in Corollary~\ref{Corollary1} leads to
  \[
  \|\Phi(t,s)X_s\|^2_{\rm ms}
  \geq
  \frac{1}{d}e^{-2\Gamma (t-s)}\|X_s\|^2_{\rm ms} \quad \text{for } (t,s)\in\R_{\geq}^2\,,
  \]
  which proves $(-\infty,\Gamma)\subset \rho$.\\
  \emph{Step 2.} It will be shown that for any $t\in\R$, the set
  \[
  \big\{S_\gamma(t): \gamma\in \rho\cap (-\Gamma-1,\Gamma+1)\big\}
  \]
  consists of at most $d(d+1)+1$ elements. Suppose the contrary, i.e.,~there exist $n+1$ numbers $\gamma_0<\gamma_{1}<\dots<\gamma_{n}$ in $\rho\cap (-\Gamma-1,\Gamma+1)$, where $n> d(d+1)$, such that
  \[
  S_{\gamma_i}(t)\not= S_{\gamma_j}(t)\qquad\hbox{for } i\not=j.
  \]
  Then by Lemma~\ref{Lemma_2},
  \begin{equation*}
  S_{\gamma_0}(t)\subsetneq S_{\gamma_{1}}(t) \subsetneq \dots\subsetneq S_{\gamma_{n}}(t).
  \end{equation*}
  Thus, there exist $X_t^1,\dots,X_t^n$ such that
  \[
  X_t^i\in S_{\gamma_i}(t),\quad X_t^i\not\in S_{\gamma_{i-1}}(t)\qquad\hbox{for } i\in\{1,\dots,n\}\,.
  \]
  By definition of the $\gamma_i$, there exist $K,\alpha>0$ and complementary subspaces $U_{\gamma_i}(t)$ such that $\mathfrak X_t = U_{\gamma_i}(t)\oplus S_{\gamma_i}(t)$ and
  \begin{equation}\label{Eq21}
  \|\Phi_{\bar t,t}X_t\|_{\rm ms}\leq K e^{(\gamma_i-\alpha)(\bar t-t)}\|X_t\|_{\rm ms}\qquad\hbox{for } X_t\in S_{\gamma_i}(t) \mbox{ and } \bar t\geq t
  \end{equation}
  and
  \begin{equation}\label{Eq22}
  \|\Phi_{\bar t,t}X_t\|_{\rm ms}\geq \frac{1}{K} e^{(\gamma_i+\alpha)(\bar t-t)}\|X_t\|_{\rm ms}\qquad\hbox{for } X_t\in U_{\gamma_i}(t)\mbox{ and }t\geq t\,.
  \end{equation}
  Since $\R^{\frac{d(d+1)}{2}}\times\R^{\frac{d(d+1)}{2}}$ is $d(d+1)$-dimensional, it follows that there exist $k\leq n$ and  $\alpha_1$, $\dots$, $\alpha_{k-1}$ with  $\alpha_1^2+\dots+\alpha_{k-1}^2 \not= 0$ and
  \begin{equation}\label{Eq6}
  \pi_t X_t^k=\alpha_1\pi_t X_t^1+\dots+\alpha_{k-1}\pi_t X^{k-1}_t.
  \end{equation}
  Consequently, by Corollary \ref{Corollary1},
  \begin{align*}
  \|\Phi_{\bar t,t}X_t^k\|^2_{\rm ms}
  &=
  \sum_{i=1}^dV_{i,i}(\bar t, t) (\pi_tX_t^k)=
  \sum_{i=1}^d\sum_{j=1}^{k-1}\alpha_jV_{i,i}(\bar t, t) (\pi_tX_t^j)\\
  &\leq
  \sum_{j=1}^{k-1}|\alpha_j|\sum_{i=1}^d\sum_{j=1}^{k-1}V_{i,i}(\bar t, t) (\pi_tX_t^j)\\
  &=
  \sum_{j=1}^{k-1}|\alpha_j|\left(\|\Phi_{\bar t,t}X_t^1\|^2_{\rm ms}+\dots+\|\Phi_{\bar t,t}X_t^{k-1}\|^2_{\rm ms}\right).
  \end{align*}
  By definition of $\gamma_i$ and \eqref{Eq21}
  \[
  \|\Phi_{\bar t,t}X_t^k\|^2_{\rm ms}
  \leq
  (k-1)K\left(\sum_{j=1}^{k-1}|\alpha_j|\right)\left(\sum_{j=1}^{k-1}\|X_t^j\|^2_{\rm ms}\right)e^{(\gamma_{k-1}-\alpha)(\bar t-t)}.
  \]
  Hence, it follows by \eqref{Eq21} and \eqref{Eq22}  that $X_t^k\in S_{\gamma_{k-1}}(t)$, which leads to a contradiction.\\
  \emph{Step 3.} As proved in Step~2, for $t\in\R$, let $S_0(t) \subsetneq S_1(t) \subsetneq \dots \subsetneq S_n(t)$ with $n\leq d(d+1)$
  satisfy
  \[
  \big\{S_\gamma(t): \gamma\in \rho\cap  (-\Gamma-1,\Gamma+1)\big\}=\{S_0(t),S_1(t),\dots,S_n(t)\big\}.
  \]
  By Step 1, it follows that $S_0(t)=\{0\}$ and $S_n(t)=\mathfrak X_t$. For each $i\in\{0,\dots,n\}$, define
  \[
  \mathcal I_i:=\big\{\gamma \in \rho\cap (-\Gamma-1,\Gamma+1): S_\gamma(t)=S_i(t)\big\}.
  \]
  It will be  shown that $\mathcal I_i=(b_i,a_{i+1})$, where
  $$
  b_{i}=\inf\{\gamma:\gamma \in \mathcal I_i\} \quad \mbox{and} \quad a_{i+1}=\sup\{\gamma:\gamma \in \mathcal I_i\} \qquad \text{for } i \in\{0,\dots,n\}\,.
  $$
  First, let $\gamma\in \mathcal I_i$  be arbitrary. By the definition of $\mathcal I_i$, there exist $K,\alpha>0$ and a decomposition $\mathfrak X_t=U(t)\oplus S_{i}(t)$ and
  \begin{equation*}
  \|\Phi_{\bar t,t}X_t\|_{\rm ms}\leq K e^{(\gamma-\alpha)(\bar t-t)}\|X_t\|_{\rm ms}\qquad\hbox{for } X_t\in S_{i}(t) \mbox{ and } \bar t\geq t
  \end{equation*}
  and
  \begin{equation*}
  \|\Phi_{\bar t,t}X_t\|_{\rm ms}\geq \frac{1}{K} e^{(\gamma+\alpha)(\bar t-t)}\|X_t\|_{\rm ms}\qquad\hbox{for } X_t\in U(t) \mbox{ and } \bar t\geq t\,.
  \end{equation*}
  This implies that $(\gamma-\alpha,\gamma+\alpha) \subset  \mathcal I_i$, so $\mathcal I_i$ is open.   It can be shown similarly  that $\mathcal I_i$ is
  connected. Hence $\mathcal I_i=(a_i,b_i)$. Combining this result and Step 1 gives
  \[
  \rho=(-\infty, a_1)\cup (b_1,a_2)\cup\dots\cup (b_{n-1},a_n)\cup (b_n,\infty),
  \]
  which implies that
  \[
  \Sigma= [a_1,b_1]\cup\dots\cup [a_{n},b_n].
  \]
  To conclude the proof, the filtration corresponding to the spectral intervals is constructed as follows:  for $t\in \R$,
  $V_0(t):=\{0\}, V_n(t):=\mathfrak X_t$, and
  \[
  V_i(t):=S_{\gamma}(t),\quad \mbox{where }\gamma\in (b_i,a_{i+1}), \quad i\in\{1,\dots,n-1\}\,.
  \]
  Due to Lemma~\ref{Lemma_1}, the definition of $V_i$ is independent of $\gamma \in (b_i,a_{i+1})$ for $i\in\{1,\dots,n-1\}$. The strict inclusion $V_i \subsetneq V_{i+1}$ for $i\in\{0,\dots,n-1\}$ follows from the construction of the open interval $(b_i,a_{i+1})$ above. Finally, the dynamical characterisation of $V_i$ follows from the definition of $(b_i,a_{i+1})$ and the definition of exponential dichotomy. This completes the proof.
\end{proof}
\section{Bifurcation of a mean-square random attractor}
A mean-square random attractor was defined in \cite{Kloeden_12_2} as the pullback attractor of the nonautonomous dynamical system formulated as a mean-square random dynamical system.

Specifically, a family $\mathcal A=\{A_t\}_{t\in\R}$ of nonempty compact subsets of $\mathfrak X$ with $A_t\subset \mathfrak X_t$ for each $t\in\R$ is called a \emph{pullback attractor} if it pullback attracts all uniformly bounded families $\mathcal D=\{D_t\}_{t\in\R}$ of subsets of $\{\mathfrak X_t\}_{t\in\R}$, i.e.,
\[
\lim_{s\to-\infty} \hbox{dist}(\phi(t,s,D_{s}), A_t) = 0.
\]
Uniformly bounded here means that there is an $R > 0$ such that $\|X\|_{\rm ms} \leq R$  for all $X \in D_t$ and $t \in \R$.

The existence of pullback attractors follows from that of an absorbing family. A uniformly bounded family $\cB=\{B_t\}_{t\in\R}$ of nonempty closed subsets of $\{\mathfrak X_t\}_{t\in\R}$ is called a \emph{pullback absorbing family} for a MS-RDS $\phi$ if for each $t\in\R$ and every uniformly bounded family $\mathcal D=\{D_t\}_{t\in\R}$ of nonempty subsets of $\{\mathfrak X_t\}_{t\in\R}$, there exists some $T=T(t,\mathcal D)\in\R^+$ such that
\[
\phi(t,s,D_s)\subseteq B_t\quad\hbox{for } s\in\R\hbox{ with } s\leq t-T.
\]

\begin{theorem}\label{PullbackAttractor}
Suppose that a MS-RDS $\phi$ has a positively invariant pullback absorbing uniformly bounded family $\mathcal B=\{\mathcal B_t\}_{t\in\R}$ of nonempty closed subsets of $\{\mathfrak X_t\}_{t\in\R}$ and that the mappings $\phi(t,s,\cdot):\mathfrak X_s\rightarrow \mathfrak X_t$ are pullback compact (respectively, eventually or asymptotically compact) for all $(t,s)\in \R^2_{\geq}$. Then, $\phi$ has a unique global pullback attractor $\mathcal A=\{\mathcal A_t\}_{t\in\R} $ with its component sets  determined by
\[
A_t=\bigcap_{s\leq t} \phi(t,s,B_s)  \qquad\hbox{for }  t\in\R.
\]
\end{theorem}

Consider the nonlinear mean-field SDE
\begin{equation} \label{MNsde}
  \rmd X_t=\left(\alpha X_t+\beta \E X_t-X_t\E X_t^2\right)\,\rmd t+ X_t  \, \rmd W_t
\end{equation}
with real-valued parameters $\alpha,\beta$. Note that the theory in Section~\ref{sec_mssde} can be easily extended to include the second moment of the solution in the equation.

This SDE has the steady state solution $\bar{X}(t) \equiv 0$. Linearising along this solution gives the
bi-linear mean-field SDE
\begin{equation} \label{Itosde}
  \rmd Z_t =  \left(\alpha Z_t+\beta \E Z_t\right)\, \rmd t+Z_t \, \rmd W_t\,.
\end{equation}

\begin{theorem} \label{dichthm}The dichotomy spectrum of the linear MS-RDS $\Phi$ generated by \eqref{Itosde} is given by
\[
\Sigma=
\left\{
\begin{array}{ll}
\{\alpha+1/2\}\cup\{\alpha+\beta\} & \quad \hbox{if } \beta>1/2\,,\\
\{\alpha+1/2\} & \quad \hbox{if } \beta\leq 1/2\,.
\end{array}
\right.
\]
\end{theorem}
\begin{proof}
Taking the expectation of two sides of \eqref{Itosde} yields that
\[
\frac{\rmd}{\rmd t}\E Z_t=(\alpha+\beta)\E Z_t,
\]
which implies that
\begin{equation}\label{Bifurcation_Eq1}
\E \Phi(t,s)Z_s=e^{(\alpha+\beta)(t-s)} \E Z_s\qquad\hbox{for } (t,s)\in \R^2_{\geq} \mbox{ and } Z_s\in \mathfrak X_s\,.
\end{equation}
Ito's formula for the function $U(x)=x^2$ then gives
\[
  \rmd Z_t^2=\left[(2\alpha+1)Z_t^2+2\beta Z_t\E Z_t\right]\,\rmd t+ 2 Z_t^2\,\rmd W_t\,.
\]
Consequently,
\[
\frac{\rmd}{\rmd t} \E Z_t^2= (2\alpha+1) \E Z_t^2+ 2\beta (\E Z_t)^2.
\]
Thus, using \eqref{Bifurcation_Eq1} for $(t,s)\in \R^2_{\geq}$ and $Z_s\in \mathfrak X_s$, it follows that
\begin{align}
\|\Phi(t,s)Z_s\|^2_{\rm ms}
&=
e^{(2\alpha+1)(t-s)} \|Z_s\|^2_{\rm ms}+2\beta\int_s^t e^{(2\alpha+1)(t-u)}(\E Z_u)^2\,\rmd u    \nonumber \\
&=
e^{(2\alpha+1)(t-s)} \|Z_s\|^2_{\rm ms}+2\beta (\E Z_s)^2\int_s^t e^{(2\alpha+1)(t-u)}e^{(2\alpha+2\beta)(u-s)}\,\rmd u  \nonumber\\
&=
e^{(2\alpha+1)(t-s)}\left(\|Z_s\|^2_{\rm ms}+2\beta  (\E Z_s)^2 \int_s^t e^{(2\beta-1)(u-s)}\,\rmd u \right)\,.\label{Bifurcation_Eq2}
\end{align}
The assertions  of the lemma will be shown for the three cases $\beta<1/2$, $\beta=1/2$ and $\beta>1/2$.\\
\emph{Case 1} ($\beta< 1/2$).
By \eqref{Bifurcation_Eq2},
\[
\|\Phi(t,s)Z_s\|^2_{\rm ms}
=
e^{(2\alpha+1)(t-s)}\left(\|Z_s\|^2_{\rm ms} + \frac{2\beta}{1-2\beta}\left(1-e^{(2\beta-1)(t-s)}\right)  (\E Z_s)^2\right),
\]
which together with the inequality $(\E X)^2 \leq \E X^2$ yields
\[
\|\Phi(t,s)Z_s\|^2_{\rm ms}
\leq
e^{(2\alpha+1)(t-s)}\left(1+\frac{2|\beta|}{(1-2\beta)}\right)\|Z_s\|^2\,,
\]
and
\[
\|\Phi(t,s)Z_s\|^2_{\rm ms}
\geq
\left\{
  \begin{array}{ll}
e^{(2\alpha+1)(t-s)}\|Z_s\|^2_{\rm ms}
 & \hbox{if } \beta\geq 0\,, \\
\frac{1}{1-2\beta}e^{(2\alpha+1)(t-s)}\|Z_s\|^2_{\rm ms} & \hbox{if } \beta<0\,.
  \end{array}
\right.
\]
This implies that $\Sigma=\{\alpha+1/2\}$.\\
\emph{Case 2} ($\beta=1/2$).
By \eqref{Bifurcation_Eq2},
\[
\|\Phi(t,s)Z_s\|^2_{\rm ms}
=
e^{(2\alpha+1)(t-s)}\left(\|Z_s\|^2_{\rm ms}  +(t-s)  (\E Z_s)^2\right),
\]
which implies that
\[
\|\Phi(t,s)Z_s\|^2_{\rm ms}
\geq
e^{(2\alpha+1)(t-s)}\|Z_s\|^2_{\rm ms}\quad \mbox{for } (t,s)\in\R_{\geq}^2\,.
\]
Let $\eps>0$ be arbitrary. Since $(t-s)\leq \frac{1}{\eps}e^{\eps (t-s)}$ for $t\geq s$  and $(\E X)^2\leq \E X^2$, it follows that
\[
\|\Phi(t,s)Z_s\|^2_{\rm ms}
\leq
\big(\textstyle 1+\frac{1}{\eps}\big)e^{(2\alpha+1+\eps)(t-s)}\|Z_s\|^2_{\rm ms}.
\]
Consequently, $\Sigma \subset \big[\alpha+\frac{1}{2}, \alpha+\frac{1}{2}+\eps\big]$. The limit $\eps\to 0$ leads to $\Sigma = \{\alpha+1/2\}$.\\
\emph{Case 3} ($\beta>1/2$). By \eqref{Bifurcation_Eq2},
\begin{equation} \label{Eq11}
\|\Phi(t,s)Z_s\|^2_{\rm ms}
=
e^{(2\alpha+1)(t-s)}\|Z_s\|^2_{\rm ms} +\frac{2\beta}{2\beta-1}\left(e^{(2\alpha+2\beta)(t-s)}-e^{(2\alpha+1)(t-s)}\right)  (\E Z_s)^2\,.
\end{equation}
Together with the inequality $(\E X)^2\leq \E X^2$,  this implies that
\[
e^{(2\alpha+1)(t-s)}\|Z_s\|^2_{\rm ms}
\leq
\|\Phi(t,s)Z_s\|^2_{\rm ms}
\leq
\left(1+\frac{2\beta}{2\beta-1}\right)e^{(2\alpha+2\beta)(t-s)}\|Z_s\|^2_{\rm ms}.
\]
Consequently, $\Sigma \subset \big[\alpha+\frac{1}{2},\alpha+\beta\big]$.
Let $\gamma \in \big(\alpha+\frac{1}{2},\alpha+\beta \big)$ be arbitrary. Choose and fix $\eps>0$ such that $(\gamma-\eps,\gamma+\eps)\subset \big(\alpha+\frac{1}{2},\alpha+\beta\big)$. The aim is to show  that $\Phi$ admits an exponential dichotomy with the growth rate $\gamma$ for the decomposition $\mathfrak X_s=U_s\oplus S_s$, where
\begin{align*}
S_s
&:=
\big\{f\in \mathfrak X_s: \E f=0\big\},\\
U_s
&:=
\big\{f\in \mathfrak X_s: f\hbox{ is independent of noise}\big\}.
\end{align*}
Obviously, any $X\in \mathfrak X_s$ can be written as $X=(X-\E X)+\E X$, with $X-\E X\in S_s$ and $\E X\in U_s$. By \eqref{Eq11}, for any $Z_s\in S_s$,
\[
\|\Phi(t,s)Z_s\|^2_{\rm ms} =e^{(2\alpha+1)(t-s)}\|Z_s\|^2_{\rm ms}\,.
\]
Now  $(\E Z_s)^2=\E Z_s^2$ for any $Z_s\in U_s$, so by \eqref{Eq11}, for $t-s\geq 1$,
\begin{align*}
\|\Phi(t,s)Z_s\|^2_{\rm ms}
&\geq
\frac{2\beta}{2\beta-1}\left(e^{(2\gamma+2\eps)(t-s)}-e^{(2\gamma-2\eps)(t-s)}\right) \|Z_s\|^2_{\rm ms}   \\
&\geq
\frac{4\eps \beta}{2\beta-1}e^{2\gamma(t-s)}\|Z_s\|^2_{\rm ms}\,.
\end{align*}
Here  the inequality $e^x\geq 1+x$ for $x\geq 0$ has been used. Thus, $\Phi$ admits an exponential dichotomy with the growth rate $\gamma$, which means that $\Sigma \subset \{\alpha+1/2\}\cup \{\alpha+\beta\}$. Considering the decomposition $S_s\oplus U_s$,  it follows that $\alpha+\frac{1}{2}, \alpha+\beta \in \Sigma$. Thus, $\Sigma =  \{\alpha+1/2\}\cup \{\alpha+\beta\}$.
This completes the proof.
\end{proof}

A globally bifurcation of pullback attractor of \eqref{MNsde} with $\beta=1$, i.e.,
\begin{equation} \label{MNsde_01}
  \rmd X_t=\left(\alpha X_t+\E X_t - X_t \E X_t^2\right)\,\rmd t+ X_t\,  \rmd W_t\,,
\end{equation}
as $\alpha$ varies, will be investigated in a series of theorems.

The  first and second moment equations  of the mean-field  SDE \eqref{MNsde_01}  are given by
\begin{align}\label{momeq1}
\frac{\rmd }{\rmd t}\E X_t  & =    (\alpha+1) \E X_t -   \E X_t  \E X_t^2\,,
\\ \label{momeq2}
\frac{\rmd }{\rmd t}\E X_t^2 & =  (2\alpha+1) \E X_t^2+ 2 (\E X_t)^2-2(\E X_t^2)^2\,,
\end{align}
where Ito's formula with $y=x^2$ was used to derive \eqref{momeq2}.  These   can be rewritten as  the system of ODEs
$$
\frac{\rmd x}{\rmd t} =x(\alpha+1 -y)\quad \mbox{and} \quad \frac{\rmd y}{\rmd t} = (2\alpha+1) y + 2x^2 - 2 y^2\,, \quad  \mbox{where } x^2 \leq y\,,
$$
which has  a steady state solution $\bar{x} =$  $\bar{y} =$  $0$ for all $\alpha$ corresponding to the zero solution $X_t \equiv 0$ of the mean-field SDE \eqref{MNsde_01}.  There also exist valid (i.e.,~with $y \geq 0$) steady state solutions $\bar{x}$  $= \pm \sqrt{(\alpha +1)/2}$ , $\bar{y} = \alpha +1$ for $\alpha > -1$ and  $\bar{x}$  $= 0$, $\bar{y} = \alpha + \frac{1}{2}$ for $\alpha \geq -\frac{1}{2}$.  It needs to be shown if there are solutions of the SDE
\eqref{MNsde_01} with these moments.

\begin{theorem}\label{PBA1}
 The MS-RDS $\phi$  generated by \eqref{MNsde_01} has a uniformly bounded positively invariant pullback absorbing family.
 \end{theorem}

\begin{proof}
Let $\alpha$ be arbitrary and  define
\[
B_t=\big\{X\in \mathfrak X_t: \|X\|_{\rm ms}\leq \sqrt{|\alpha|+2}\big\}\quad\hbox{for } t\in \R\,.
\]
Using
\[
(2\alpha+1) \E X_t^2+ 2 (\E X_t)^2-2(\E X_t^2)^2 \leq (2\alpha+3)\E X_t^2-2(\E X_t^2)^2\,,
\]
which holds since $(\E X_t)^2  \leq \E X_t^2$, the second moment equation \eqref{momeq2} gives the differential inequality
\begin{equation} \label{MNsde_inq}
\frac{\rmd }{\rmd t}\E X_t^2 \leq (2\alpha+3)\E X_t^2-2(\E X_t^2)^2
\end{equation}
Let $\mathcal D=\{D_t\}_{t\in\T}$ be a  uniformly bounded family of nonempty subsets of $\{\mathfrak X_t\}_{t\in\R}$, i.e.,  $D_t\subset \mathfrak X_t$ and there exists $R>0$ such that $\|X\|_{\rm ms} \leq R$ for all $X\in D_t$. Specifically, it will be shown that $\phi(t,s,D_s) \subset B_t$ for $t-s\geq T$, where $T$ is defined by
\begin{equation}\label{Choice1}
T:=\log\left(\frac{R^2}{|\alpha|+2}\right).
\end{equation}
Pick $X_s\in D_s$ arbitrarily and $(t,s)\in \R_{\geq}^2$ with $t-s\geq T$.  Motivated by the differential inequality \eqref{MNsde_inq}, consider the   scalar system
\begin{equation}\label{Choice2}
\dot y=(2\alpha+3) y- 2y^2,\qquad \mbox{where } y(s) \leq R^2\,.
\end{equation}
A direct computation yields
\[
  y(t)=y(s)\exp\left(\int_s^t(2\alpha+3)-2y(u)\,\rmd u\right)    \leq R^2\exp\left(\int_s^t(2\alpha+3)-2y(u)\,\rmd u\right)\,.
\]
From the definition of $T$ in \eqref{Choice1}, it follows that $\min_{s\leq u\leq t}y(u)\leq |\alpha|+2$. Furthermore, $y=0$ and $y=\alpha+\frac{3}{2}$ are stationary points of the ODE \eqref{Choice2}. For this reason, $\min_{s\leq u\leq t}y(u) \leq |\alpha|+2$ implies that $y(t) \leq |\alpha|+2$. Then from \eqref{MNsde_inq}, it follows that $y(t) \geq \|\phi(t,s)X_s\|^2_{\rm ms}$.  This means that
\begin{equation}\label{Eq15}
  \|\phi_{t,s}X_s\|^2_{\rm ms}\leq y(t)\leq |\alpha|+2\,,
\end{equation}
i.e.,~$\phi_{t,s}X_s \in B_t$ for $t-s\geq T$. Hence, $\{B_t\}_{t\in\R}$ is a pullback absorbing family for the MS-RDS $\phi$. It is clear that this family is uniformly bounded and positively invariant for the MS-RDS $\phi$.
\end{proof}

\begin{theorem}\label{PBA2}
  The MS-RDS $\phi$  generated by \eqref{MNsde_01}  has a  pullback attractor with component sets $\{0\}$ when  $\alpha < -1$.
\end{theorem}

\begin{proof}
Let $\alpha<-1$ be arbitrary. Let $\mathcal D = \{D_t\}_{t\in\R}$ be a uniformly bounded family of nonempty subsets of $\{\mathfrak X_t\}_{t\in\R}$ with $\|X\|_{\rm ms}\leq R$ for all $X\in D_t$
where $R>0$. Let $(X_s)_{s\in\R}$ be an arbitrary sequence with $X_s\in D_s$. The moment equations \eqref{momeq1}--\eqref{momeq2} can be written as
\begin{align*}
  \frac{\rmd}{\rmd t}\E [\phi(t,s)X_s]
  &=
  \E [\phi(t,s)X_s]\left(\alpha+1- \E [\phi(t,s)X_s]^2\right)
  \\ \label{mom2}
  \frac{\rmd }{\rmd t}\E[\phi(t,s)X_s]^2
  &=
  (2\alpha+1) \E [\phi(t,s)X_s]^2+ 2 (\E [\phi(t,s)X_s])^2     -2 (\E [\phi(t,s)X_s]^2)^2.
\end{align*}
Then
\[
\E [\phi(t,s)X_s]=\E X_s \exp\left(\int_s^t\alpha+1-\E X_u^2\,\rmd u\right)\,,
\]
which implies that
\[
(\E [\phi(t,s)X_s])^2\leq e^{(\alpha+1)(t-s)} R^2\,.
\]
Moreover, by the variation of constants formula,
\begin{align*}
&\E [\phi(t,s)X_s]^2\\
=&\,
e^{(2\alpha+1)(t-s)}\E X_s^2+2\int_s^t e^{(2\alpha+1)(t-u)}\Big(\E [\phi(u,s)X_s]^2
-(\E [\phi(u,s)X_s]^2)^2\Big)\,\rmd u\\
\leq&\,
e^{(2\alpha+1)(t-s)} R^2+2 R^2\int_s^t e^{(2\alpha+1)(t-u)}e^{(\alpha+1)(u-s)}\,\rmd u\\
\leq &\,
e^{(2\alpha+1)(t-s)} R^2+2 e^{(\alpha+1)(t-s)}(t-s)R^2\,,
\end{align*}
which implies that $\lim_{s\to-\infty} \|\phi(t,s,X_s)\|_{\rm ms}=0$. Thus $\{0\}$ is the pullback attractor of \eqref{MNsde_01} in this case.
\end{proof}
\begin{theorem}\label{PBA3}
  The MS-RDS $\phi$  generated by \eqref{MNsde_01}  has a nontrivial pullback attractor when $-1 < \alpha < -\frac{1}{2}$ .
\end{theorem}

\begin{remark}
 The idea for  the proof is taken from \cite[Subsection 4.1]{Kloeden_12_2},  but their result cannot be applied directly, since the Lipschitz constant of the nonlinear terms is at least $1$, and $\alpha$ is not less than $-4$.  Instead  the uniform equicontinuity of the mapping $t\mapsto \E X_t$, and then the positivity of the second moment to obtain a better estimate are used.
\end{remark}

\begin{proof}
  In order to apply Theorem~\ref{PullbackAttractor}, it needs to be shown that $\phi$ is pullback asymptotically compact, i.e.,~given a uniformly bounded family $\mathcal D=\{D_t\}_{t\in\R}$  of nonempty subsets of $\mathfrak X_t$   and  sequences $\{t_k\}_{k\in\N}$ in $(-\infty, t)$ with $t_k\to-\infty$ as $k\to\infty$ and $\{X_k\}_{k\in\N} $ with $X_k\in D_{t_k}\subset \mathfrak X_{t_k}$ for each $k\in\N$, then the subset $\{\phi(t,t_k,X_k)\}_{k\in\N}\subset \mathfrak X_t$ is relatively compact. For this purpose, let $\eps > 0$ be arbitrary. A  finite cover of $\{\phi(t,t_k,X_k)\}_{k\in\N}$ with diameter less than $\eps$ will be constructed. Choose and fix $s\in\R$ with $s<t$ such that
  \begin{equation}\label{Eq17}
  4e^{(2\alpha+1)(t-s)}(|\alpha|+2)^2<\frac{\eps^2}{2}\,,
  \end{equation}
  and define $Y^k_s:=\phi(s,t_k,X_k)$. Using \eqref{Eq15} it can be assumed without loss of generality that
  \begin{equation}\label{Eq16}
  \E (Y^k_s)^2\leq |\alpha|+2\qquad\hbox{for } k\in\N.
  \end{equation}
  For $u\in [s,t]$, let $Y_u^k:=\phi(u,s)Y_s^k$ and consider a family of functions $f^k:[s,t]\rightarrow \R$ defined by  $f^k(u) := \E Y_u^k$. By Ito's formula,
  \[
  Y^k_t=e^{\alpha(t-s)}Y_s^k+\int_s^t e^{\alpha(t-u)}\left(\E Y^k_u-Y^k_u\E (Y^k_u)^2\right)\, \rmd u+\int_s^t e^{\alpha(t-u)}Y^k_u\,\rmd W_u,
  \]
  from which it can be shown that that $(f^k)_{k\in\N}$ is a uniformly equicontinuous sequence of functions. By the Arzel\`{a}--Ascoli theorem, for $\delta:=\eps\sqrt{\frac{\alpha}{2}+\frac{1}{4}}$, there exists a finite index set $J(\delta)\subset \N$ such that  for any $k\in \N$ there exists $n_k\in J(\delta)$ for which
  \begin{equation}\label{Ascoli}
  \sup_{u\in [s,t]} |\E Y_u^k-\E Y_u^{n_k}|<\delta\,.
  \end{equation}
  To conclude the proof of asymptotic compactness, it is sufficient to show the inequality $\|Y_t^k-Y_t^{n_k}\|_{\rm ms} \leq \eps$. Indeed, for $u\in [s,t]$, a direct computation gives
  \begin{align*}
  \frac{\rmd }{\rmd t}\E (Y_u^k-Y_u^{n_k})^2
  &=
  (1+2\alpha)\E (Y_u^k-Y_u^{n_k})^2+2(\E Y_u^k- \E Y_u^{n_k})^2-2(\E (Y_u^k)^2)^2\\
  &
  \hspace{10mm} -2(\E (Y_u^{n_k})^2)^2+ 2\E Y_u^kY_u^{n_k}(\E (Y_u^k)^2+\E (Y_u^{n_k})^2)\,.
  \end{align*}
  Note that
  \[
  2(\E (Y_u^k)^2)^2+2(\E (Y_u^{n_k})^2)^2- 2\E Y_u^kY_u^{n_k}(\E (Y_u^k)^2+\E (Y_u^{n_k})^2)\geq 0\,.
  \]
  Hence, by  the variation of constant formula,
  \begin{displaymath}
  \|Y_t^k-Y_t^{n_k}\|^2_{\rm ms}
  \leq
  e^{(2\alpha+1)(t-s)}\|Y_s^k-Y_s^{n_k}\|^2_{\rm ms}+2\int_s^t e^{(2\alpha+1)(t-u)}(\E Y_u^k- \E Y_u^{n_k})^2\,\rmd u\,,
  \end{displaymath}
  which together with \eqref{Eq16} and \eqref{Ascoli} implies that
  \begin{displaymath}
  \|Y_t^k-Y_t^{n_k}\|^2_{\rm ms}
  \leq
  4e^{(2\alpha+1)(t-s)}(|\alpha|+2)^2+\frac{2\delta^2}{2\alpha+1}
  <
  \frac{\eps^2}{2}+\frac{\eps^2}{2}=\eps^2\,.
  \end{displaymath}
  Here  \eqref{Eq17} was used  to obtain the preceding inequality.  Thus the set $\{\phi(t,t_k,X_k)\}_{k\in\N}$ is covered by the finite union of open balls with radius $\eps$ centered at $\phi(t,t_k,X_k)$, where $k\in J(\lambda)$, i.e.,~is totally bounded.

  Hence the MS-RDS is asymptotically compact and by  Theorem~\ref{PullbackAttractor} has a pullback attractor with component subsets $\{A_t\}_{t\in\R}$ that contain the zero solution.  It remains to show that the sets $A_t$ also contain other points. Consider a uniformly bounded family of bounded sets defined by
  $$
  D_t := \left\{X \in \mathfrak{X}_t  :   \E X   = \sqrt{(\alpha +1)/2}\,  , \;\E X^2 =   \alpha +1 \right\}.
  $$
  Recall that these values are a steady state solution of the moment equations \eqref{momeq1}--\eqref{momeq2}. Hence $\phi(t,s,X_s) \in D_t $ for all $t > s$ when $X_s \in D_s $. Then by pullback attraction
  \[
  \hbox{dist}(\phi(t,s,X_{s}), A_t) \leq  \hbox{dist}(\phi(t,s,D_{s}), A_t)\to 0\quad\hbox{as}\quad s\to-\infty\,.
  \]
  The convergence here is mean-square convergence, and $\E \phi(t,s,X_{s})^2 \equiv  \alpha +1$ for all~$t > s$. Thus, there exists a random variable $X^*_t \in$
  $A_t \cap D_t$ for each $t \in \R$, i.e.,~the pullback attractor is nontrivial. Note that by arguments in \cite{Kloeden_11_3} (see also \cite{Kloeden_11_2}), it follows that there is in fact an  entire solution $\bar{X}_t \in A_t \cap D_t$ for all $t \in \R$.
\end{proof}
%
%
%
%
%

%


%
%
\end{document}